\documentclass[oneside,english]{amsart}
\usepackage[T1]{fontenc}
\usepackage[latin9]{inputenc}
\usepackage{geometry}
\usepackage{amstext}
\usepackage{amsthm}
\usepackage{amssymb}

\theoremstyle{definition}

\newtheorem{exm}{Example}

\DeclareMathOperator{\Poly}{Poly}

\usepackage{hyperref}
\hypersetup{colorlinks,%
citecolor=red,%
linkcolor=blue,%
}
\usepackage{color}

\makeatletter
\numberwithin{equation}{section}
\numberwithin{figure}{section}
\theoremstyle{plain}
\newtheorem{thm}{\protect\theoremname}
  \theoremstyle{plain}
\newtheorem{lemma}[equation]{Lemma}

\makeatother

\usepackage{babel}

\providecommand{\theoremname}{Theorem}

\begin{document}
\title{Simple proofs and expressions for the restricted partition function and its polynomial part}

\author{Sinai Robins and Christophe Vignat}
\address{Sinai Robins, Departamento de Ciência da Computação, Instituto de Matemática e Estatística,
Universidade de São Paulo, Brasil}
\email{sinai.robins@gmail.com}
\address{Christophe Vignat, L.S.S., CentraleSupelec, Universit\'{e} Paris-Sud, Orsay, France and Department of Mathematics, 
Tulane University, New Orleans, USA}
\email{christophe.vignat@u-psud.fr}
\begin{abstract}
In this note, we provide a simple derivation of expressions for the restricted partition function and its polynomial part. 
Our proof relies on elementary algebra on rational functions and a lemma that expresses the polynomial part as an average of the partition function.
\end{abstract}
\maketitle

\section{Introduction}

We recall that the restricted partition function is defined by 
\[
p_{\mathbf{a}}\left(n\right)=\#\left\{ \left(x_{1},\dots,x_{r}\right)\in\mathbb{N}^{r};\sum_{i=1}^{r}a_{i}x_{i}=n\right\},
\]
for any fixed integer vector
$\mathbf{a}=\left(a_{1},\dots,a_{r}\right)\in \mathbb{N}^{r}$.
It is known that the restricted partition function has the following form  \cite[Theorem 1.8]{Beck}:
\begin{align*}
 p_{\bf a}(n)        &= \Poly_{\bf a}(n)  \\
&+ s_{-n}(a_2, a_3, \dots, a_r; a_1) + s_{-n}(a_1, a_3, a_4, \dots, a_r; a_2)  + 
\cdots +
s_{-n}(a_1, a_2, a_3, \dots, a_{r-1}; a_r),
\end{align*}
where we use here the Fourier-Dedekind sums, defined by
\begin{equation}  \label{fourierdedsum}
s_n \left( a_1, a_2, \dots, a_m; b \right) := \frac 1 b \sum_{j=1}^{b-1} \frac{ \xi_b^{ jn } }{ \left( 1 - \xi_b^{ j a_1 } \right) \left( 1 - \xi_b^{ j a_2 } \right) \cdots \left( 1 - \xi_b^{ j a_m } \right) },
\end{equation}
and where $\Poly_{\bf a}(n)$ is known as the polynomial part of  $p_{\bf a}(n)$ (see \cite{Beck}, equation 8.4, and exercise 14.3 for a Barnes-Bernoulli type formulation of  $\Poly_{\bf a}(n)$).
Throughout, $\xi_b :=   e^{\frac{2 \pi i}{b}}$ denotes the $b^{th}$ root of unity. 
We also recall that there exist $D=lcm\left(a_{1},\dots,a_{r}\right)$ polynomials $\left\{ q_{k}\left(n\right)\right\} _{0\le k\le D-1}$, called the constituent polynomials, such that for $n\equiv k \pmod D$, we have 
$p_{\mathbf{a}}\left(n\right)=q_{k}\left(n\right)$.
Indeed, we have, for $n \equiv k \pmod D$:
\begin{align}   \label{constituent.Fourier} 
q_{k}(n)        &= \Poly_{\bf a}(n)      \\   
&+ s_{-k}(a_2, a_3, \dots, a_r; a_1) + s_{-k}(a_1, a_3, a_4, \dots, a_r; a_2)  + 
\cdots +
s_{-k}(a_1, a_2, a_3, \dots, a_{r-1}; a_r), \nonumber
\end{align}
and we see that all of the constituent polynomials $q_k(n)$ have the same polynomial part 
$\Poly_{\bf a}(n)$, which is independent of $k$. 
The main result in the recent paper by Cimpoea\c{s} and Nicolae \cite{Cimpoeas} is the following simplified expression for the restricted partition function.  
\begin{thm}  \label{CimpoeasAndNicolae}
The restricted partition function may be expressed as
\begin{equation}
p_{\mathbf{a}}\left(n\right)=\sum_{\stackrel{\mathbf{j}\in\mathbf{J}}{a_{1}j_{1}+\dots+a_{r}j_{r}\equiv \ n \mod D}}
\binom{\frac{n-a_{1}j_{1}-\dots-a_{r}j_{r}}{D}+r-1}{r-1}\label{eq:1}
\end{equation}
\end{thm}
where $D=lcm\left(a_{1},\dots,a_{r}\right)$, and where the summation index runs over the box 
\[
\mathbf{J}=\left\{ \mathbf{j}=\left(j_{1},\dots,j_{r}\right);0\le j_{1}\le\frac{D}{a_{1}}-1,\dots,0\le j_{r}\le\frac{D}{a_{r}}-1\right\}.
\]
Here we give a simple proof of Theorem \ref{CimpoeasAndNicolae}, and we also 
recover the following 
expression for the polynomial part $\text{Poly}_{\mathbf{a}}\left(n\right)$ of the restricted partition function, in the same spirit as Cimpoeas and Nicolae \cite{Cimpoeas}.

\begin{thm}
The polynomial part of the restricted partition function may be expressed as
\begin{equation}
\Poly_{\mathbf{a}}\left(n\right)=\frac{1}{D}\sum_{\mathbf{j}\in\mathbf{J}}\binom{\frac{n-a_{1}j_{1}-\dots-a_{r}j_{r}}{D}+r-1}{r-1}.\label{eq:2}
\end{equation}
\end{thm}
This simple expression was also derived recently in \cite{Dilcher} from properties of Bernoulli numbers, and appeared several times in the literature under different but more complicated forms (see the introduction of \cite{Cimpoeas} for more details).

\section{A short proof}
Using a technique by F. Breuer in \cite{Breuer}, we give a short
proof of (\ref{eq:1}); then we deduce (\ref{eq:2}) from the identity in Lemma \ref{lem:The-polynomial-part} below. First, following \cite{Breuer}, the generating function
\[
\sum_{n\ge0}p_{\mathbf{a}}\left(n\right)z^{n}=\prod_{k=1}^{r}\frac{1}{1-z^{a_{k}}}
\]
is easily  transformed by forcing each term in the product to have the same
denominator $1-z^{D}:$
\[
\frac{1}{1-z^{a_{k}}}=\frac{\frac{1-z^{D}}{1-z^{a_{k}}}}{1-z^{D}}=\frac{\sum_{j_{k}=0}^{\frac{D}{a_{k}}-1}z^{a_{k}j_{k}}}{1-z^{D}}
\]
so that
\[
\sum_{n\ge0}p_{\mathbf{a}}\left(n\right)z^{n}=\frac{1}{\left(1-z^{D}\right)^{r}}\sum_{\mathbf{j}\in\mathbf{J}}z^{a_{1}j_{1}+\dots+a_{r}j_{r}}.
\]
Expanding the denominator
\[
\frac{1}{\left(1-z^{D}\right)^{r}}=\sum_{p=0}^{+\infty}\binom{p+r-1}{r-1}z^{pD}
\]
gives
\[
\sum_{n\ge0}p_{\mathbf{a}}\left(n\right)z^{n}=\sum_{p\ge0,\mathbf{j}\in\mathbf{J}}\binom{p+r-1}{r-1}z^{pD+a_{1}j_{1}+\dots+a_{r}j_{r}}.
\]
Identifying the coefficient of $z^{n}$ in the right-hand side gives
\[
p_{\mathbf{a}}\left(n\right)=\sum_{\text{\textbf{j}\ensuremath{\in\textbf{J}}}}\binom{\frac{n-a_{1}j_{1}-\dots-a_{r}j_{r}}{D}+r-1}{r-1},
\]
where the summation set consists of all indices $\mathbf{j}$ such
that 
\[
pD+a_{1}j_{1}+\dots+a_{r}j_{r}=n,
\]
and all integers $p\in\mathbb{Z}$ (note that we can allow $p<0$
since in this case the corresponding binomial coefficient $\binom{p+r-1}{r-1}$
vanishes). This is equivalent to the condition
\[
a_{1}j_{1}+\dots+a_{r}j_{r}\equiv n\mod D
\]
and gives (\ref{eq:1}).

\medskip
The next Lemma tells us that we can average these quasi-polynomials to get a simplified expression for the polynomial part $\Poly_{\bf a}(n)$. 
 
\begin{lemma}
\label{lem:The-polynomial-part}  
The polynomial part $\Poly_{\bf a}(n)$ of the restricted partition function  $p_{\bf a}(n)$ is the following average:
$$
\Poly_{\bf a}(n) = \frac{1}{D} \sum_{k=0}^{D-1} q_k(n).
$$
\end{lemma}
\begin{proof}
The main observation is that when we average over a complete residue system, the Fourier-Dedekind sums vanish. For example, for the first Fourier-Dedekind sum, averaging over  $k \pmod {D}$, the Fourier-Dedekind sum $s_{-k}(a_2, a_3, \dots, a_r; a_1)$ vanishes, as follows:

\begin{align} \label{vanishing}
&   \frac{1}{D} \sum_{k=0}^{D-1}  s_{-k}(a_2, a_3, \dots, a_r; a_1) = 
\frac{1}{D} \sum_{k=0}^{D-1}
 \frac 1 a_1 \sum_{j=1}^{a_1 -1} \frac{ \xi_{a_1}^{ jk} }{ \left( 1 - \xi_{a_1}^{ j a_2 } \right) \left( 1 - \xi_{a_1}^{ j a_3 } \right) \cdots \left( 1 - \xi_{a_1}^{ j a_r } \right) } \\
 &=
 \frac 1 a_1 \sum_{j=1}^{a_1 -1} \frac{             \frac{1}{D} \sum_{k=0}^{D-1}       \xi_{a_1}^{ jk} }
 { \left( 1 - \xi_{a_1}^{ j a_2 } \right) \left( 1 - \xi_{a_1}^{ j a_3 } \right) \cdots \left( 1 - \xi_{a_1}^{ j a_r } \right) } 
 = 0,
\end{align}
by the orthogonality relations, because $a_1 | D$, and so we are summing the $a_1$'th roots of unity over a multiple of the complete residue system mod $a_1$. 
Therefore, using \eqref{constituent.Fourier} and the fact that the polynomial part $\Poly_{\bf a}(n)$ is independent of $k$, we average the constituent polynomials  $q_k(n)$ over a complete residue system $k$ mod $D$ to get:
\begin{align*}
& \frac{1}{D} \sum_{k=0}^{D-1} q_k(n)  = \frac{1}{D} \sum_{k=0}^{D-1} \text{Poly}_{\bf a}(n)  
+  \frac{1}{D} \sum_{k=0}^{D-1}     s_{-k}(a_2, a_3, \dots, a_d; a_1)  \\
&+   \frac{1}{D} \sum_{k=0}^{D-1}    s_{-k}(a_1, a_3, a_4, \dots, a_d; a_2)  
+  \cdots +
 \frac{1}{D} \sum_{k=0}^{D-1}   s_{-k}(a_1, a_2, a_3, \dots, a_{d-1}; a_d) \\
 &= \text{Poly}_{\bf a}(n),
\end{align*}
by the vanishing argument of equation \ref{vanishing}.
\end{proof}

\noindent
Using this Lemma, we may now deduce from (\ref{eq:1}) that for $0\le k\le D-1,$
\[
p_{\mathbf{a}}\left(Dn+k\right)=\sum_{\stackrel{\mathbf{j}\in\mathbf{J}}{a_{1}j_{1}+\dots+a_{r}j_{r}\equiv k\thinspace\mod D}}\binom{\frac{Dn+k-a_{1}j_{1}-\dots-a_{r}j_{r}}{D}+r-1}{r-1}
\]
so that
\[
q_{k}\left(n\right)=p_{\mathbf{a}}\left(n\right)_{n\equiv k\mod D}=\sum_{\stackrel{\mathbf{j}\in\mathbf{J}}{a_{1}j_{1}+\dots+a_{r}j_{r}\equiv k\thinspace\mod D}}\binom{\frac{n-a_{1}j_{1}-\dots-a_{r}j_{r}}{D}+r-1}{r-1}.
\]
Hence, by Lemma \ref{lem:The-polynomial-part},
\begin{align*}
\text{Poly}_{\mathbf{a}}\left(n\right) & =\frac{1}{D}\sum_{k=0}^{D-1}\sum_{\stackrel{\mathbf{j}\in\mathbf{J}}{a_{1}j_{1}+\dots+a_{r}j_{r}\equiv k\thinspace\mod D}}\binom{\frac{n-a_{1}j_{1}-\dots-a_{r}j_{r}}{D}+r-1}{r-1}.
\end{align*}
Since the sum is over a complete residue system, this proves (\ref{eq:2}).

\begin{exm}
We let $\mathbf{a}=(2,3)$, so that we have
\[
p_{\left\{ 2,3\right\} }\left(n\right)=\sum_{\stackrel{0\le k\le2,0\le l\le1}{2k+3l\equiv n\mod6}}\left(\frac{n-2k-3l}{6}+1\right)
\]
so that
\[
p_{\left\{ 2,3\right\} }\left(6n\right)=n+1
\]
and
\[
q_{0}\left(n\right)=p_{\left\{ 2,3\right\} }\left(n\right)_{\vert n\equiv0\mod6}=\frac{n}{6}+1.
\]
Similarly,
\[
p_{\left\{ 2,3\right\} }\left(6n+1\right)=n
\]
and
\[
q_{1}\left(n\right)=p_{\left\{ 2,3\right\} }\left(n\right)_{\vert n\equiv1\mod6}=\frac{n-1}{6}
\]
and the other values are
\[
q_{2}\left(n\right)=p_{\left\{ 2,3\right\} }\left(n\right)_{\vert n\equiv2\mod6}=\frac{n-2}{6}+1,\thinspace\thinspace q_{3}\left(n\right)=p_{\left\{ 2,3\right\} }\left(n\right)_{\vert n\equiv3\mod6}=\frac{n-3}{6}+1
\]
and
\[
q_{4}\left(n\right)=p_{\left\{ 2,3\right\} }\left(n\right)_{\vert n\equiv4\mod6}=\frac{n-4}{6}+1,\thinspace\thinspace q_{5}\left(n\right)=p_{\left\{ 2,3\right\} }\left(n\right)_{\vert n\equiv5\mod6}=\frac{n-5}{6}+1.
\]
We can now check the result of Lemma \ref{lem:The-polynomial-part}: the polynomial
part is
\[
\frac{1}{6}\left(q_{0}\left(n\right)+\dots+q_{5}\left(n\right)\right)=\frac{n}{6}+\frac{5}{12}
\]
which matches the general formula for the polynomial part in two dimensions (\cite{Beck}, equation 1.7): 
\[
\text{Poly}_{\left\{ a,b\right\} }\left(n\right)=\frac{n}{ab}+\frac{1}{2a}+\frac{1}{2b}.
\]
\hfill $\square$
\end{exm}


\end{document}